\documentclass[11pt]{amsart}

\usepackage{amsmath}
\usepackage{amsthm}
\usepackage{amssymb}
\usepackage{mathrsfs}
\usepackage{verbatim}
\usepackage{mathtools}
\DeclarePairedDelimiter{\ceil}{\lceil}{\rceil}

\usepackage{xcolor}

{%
\setcounter{enumi}{0}

\begin{enumerate}}%
{\end{enumerate} }

{%
\setcounter{enumi}{0}

\begin{enumerate}}%
{\end{enumerate} }

 \newcommand{\grad}{\nabla}
\newcommand{\rn}{\mathbb{R}^{n}}
\newcommand{\re}{\mathbb{R}}

\usepackage[colorlinks]{hyperref}

\newtheorem{theorem}{Theorem}[section]
\newtheorem{lemma}[theorem]{Lemma}
\newtheorem{proposition}[theorem]{Proposition}
\newtheorem{corollary}[theorem]{Corollary}
\newtheorem*{theorem*}{Theorem}

\theoremstyle{definition}
\newtheorem{definition}[theorem]{Definition}
\newtheorem{remark}[theorem]{Remark}
\newtheorem{example}[theorem]{Example}

\numberwithin{equation}{section}

\allowdisplaybreaks

\begin{document}
		
\title[Fractional Moser-Trudinger Inequality]{Remarks on the Fractional Moser-Trudinger Inequality}

\maketitle
\centerline{\scshape  Firoj Sk}

\medskip 
{\footnotesize
  \centerline{Indian Institute of Technology,  Kanpur, India.}
 \centerline{E-mail: firoj@iitk.ac.in}
 }
 
 \smallskip

\begin{abstract}
    In this article, we study the connection between the fractional Moser-Trudinger inequality and the fractional $\left(\frac{kp}{p-1},p\right)$-Poincar\'e type inequality for any Euclidean domain and discuss the sharpness of this inequality whose analogous results are well known in the local case. We further provide sufficient conditions on domains for fractional $(q,p)$-Poincar\'e type inequalities to hold.  We also derive Adachi-Tanaka type inequalities in the non-local setting.
\end{abstract}

\smallskip

\smallskip

\date{}

\section{Introduction}\label{section:introduction}
The classical Sobolev continuous embedding asserts that for a bounded domain $\Omega$ in $\rn$
\[
W^{1,p}_0(\Omega)\xhookrightarrow{}L^{q}(\Omega)\;\;\text{ for }1\leq q\leq p^*,
\]
where $p^*=\frac{np}{n-p}$ is the Sobolev critical exponent, $1\leq p<n$ and the spaces $L^p(\Omega)$, $W^{1,p}_0(\Omega)$ denote the usual Lebesgue space and Sobolev space respectively. When $p=n$ then the embedding 
\[
W^{1,n}_0(\Omega)\xhookrightarrow{}L^{\infty}(\Omega)
\]
is not true, as one can see by taking the function $u(x)=\log(1-\log|x|)$, $x\in B(0,1)\subset\re^2$. In this direction, Trudinger \cite{Trudinger} proved that there exists $\alpha>0$ such that 
\begin{equation}\label{trudinger}
     \sup_{\substack{u\in\ W_0^{1,n}(\Omega)\\ ||\grad u||_{L^n(\Omega)}\leq 1}}\int_{\Omega}e^{\alpha|u(x)|^{\frac{n}{n-1}}}\;dx<\infty.
\end{equation}
Moser, in \cite{moser}, was able to give the precise value of the optimal constant $\alpha_n=n\;\omega_{n-1}^{\frac{1}{n-1}}$ where $\omega_{n-1}$ denotes the $(n-1)$-dimensional Hausdorff measure of the unit sphere $\mathbb{S}^{n-1}$ in $\rn$, for which the inequality \eqref{trudinger} is sharp, in the sense that it is true for each $\alpha\in[0,\alpha_n]$ and fails for $\alpha>\alpha_n$. The sharpened inequality \eqref{trudinger} is usually known as the Moser-Trudinger inequality. Moreover, if we enlarge the integrand in the Moser-Trudinger inequality \eqref{trudinger} by a suitable Borel measurable function, then it becomes infinite. Indeed
\begin{equation}\label{borel}
   \sup_{\substack{u\in\ W_0^{1,n}(\Omega)\\ ||\grad u||_{L^n(\Omega)}\leq 1}}\int_{\Omega}f(|u|)e^{\alpha_n |u(x)|^{\frac{n}{n-1}}}\;dx=\infty,  
\end{equation}
where $f:[0,\infty)\to[0,\infty)$ is any Borel measurable function such that $f(t)\to\infty$ as $t\to\infty.$
 In a seminal paper of Mancini-Sandeep (see, \cite{Man-San}), it was proved that for a simply connected domain in $\mathbb{R}^2$, the Moser-Trudinger inequality \eqref{trudinger} holds true if and only if the classical Poincar\'e inequality does. This equivalence result has been extended by Battaglia and Mancini \cite{Bat-Man} for any domain $\Omega\subset\rn$, $n\geq 3$. When $\Omega=\rn$, Adachi and Tanaka \cite{Adachi-Tanaka} proved that the following Moser-Trudinger inequality, which usually called a Adachi-Tanaka type Moser-Trudinger inequality
\begin{equation}\label{Adachi-Tanaka}
  \sup_{\substack{u\in W^{1,n}(\rn)\\||\nabla u||_{L^n(\rn)}\leq 1}}\frac{1}{||u||^n_{L^n(\rn)}}\int_{\rn}\Phi_n({\alpha|u(x)|^{\frac{n}{n-1}}})\;dx
    \begin{cases}<\infty, & 0<\alpha<\alpha_{n},\\
    =\infty, & \alpha\geq\alpha_{n},
    \end{cases}  
\end{equation}
where $\Phi_n(t):=e^t-\sum_{j=0}^{n-2}\frac{t^j}{j!}.$ It is noticeable that the above supremum is infinite at the critical level $\alpha=\alpha_n$ which is quite a surprise in this case. However, if we consider the full Sobolev norm instead of the Dirichlet norm, then the Adachi-Tanaka type Moser-Trudinger inequality \eqref{Adachi-Tanaka} is different. In this context, Ruf \cite{Ruf} $(n=2)$ and Li-Ruf \cite{Li-Ruf} ($n\geq 2$) proved that the following inequality holds true:
\begin{equation*}
 \sup_{\substack{u\in W^{1,n}(\rn)\\||u||_{W^{1,n}(\rn)}\leq 1}}\int_{\rn}\Phi_n({\alpha|u(x)|^{\frac{n}{n-1}}})\;dx
 \begin{cases}<\infty, & 0<\alpha\leq\alpha_{n},\\
 =\infty, & \alpha>\alpha_{n},
 \end{cases}
\end{equation*}
where $||u||_{W^{1,n}(\rn)}=\left(||u||_{L^n(\rn)}^n+||\nabla u||^n_{L^n(\rn)}\right)^{1/n}$ is the full Sobolev norm on $W^{1,n}(\rn)$. The Moser-Trudinger inequality has been interesting research topics for many authors. There is a vast literature available on the study of the Moser-Trudinger inequality. We refer to \cite{Adachi-Tanaka},  \cite{Adams}, \cite{Adi-San}, \cite{Cal-Ruf}, \cite{Gyu-Roy}, \cite{Li-Ruf}, \cite{Ruf} as some of the relevant works on the Moser-Trudinger inequality.
\smallskip

In this paper, we shall be concerned about such type of results in the case of the fractional Sobolev spaces. For $s\in(0,1),\;p\geq 1$ and let $\Omega\subseteq\rn$ be any open set, the fractional Sobolev space $W^{s,p}(\Omega)$ is defined as
  $$
  W^{s,p}(\Omega):=\bigg\{u\in L^p(\Omega):\int_{\Omega}\int_{\Omega}\frac{|u(x)-u(y)|^p}{|x-y|^{n+sp}}dxdy<\infty\bigg\};
  $$ 
endowed with the norm $||u||_{s,p,\Omega}:=(||u||^p_{L^p(\Omega)}+[u]^p_{s,p,\Omega})^{1/p}$, where  $$[u]_{s,p,\Omega}:=\Bigg(\int_{\Omega}\int_{\Omega}\frac{|u(x)-u(y)|^p}{|x-y|^{n+sp}}dxdy\Bigg)^{1/p}$$ is the Gagliardo semi-norm. The spaces $\tilde{W}^{s,p}_0(\Omega)$ and $W^{s,p}_0(\Omega)$ are defined as the closure of the space $(C^\infty_c(\Omega))$ with the norm $||\cdot||_{s,p,\rn}$ and $||\cdot||_{s,p,\Omega}$ respectively. If $\Omega$ is a bounded Lipschitz set, then we have 
\begin{equation*}
    \tilde{W}^{s,p}_0(\Omega)=\{u\in W^{s,p}(\rn):u=0\text{ in }\rn\setminus\Omega\}.
\end{equation*}
Moreover, if $sp\neq 1$ then $\tilde{W}^{s,p}_0(\Omega)=W^{s,p}_0(\Omega)$ in the sense of equivalent norm. For more details we refer to  \cite{Bras-Lin-Par}, \cite{Bras}, \cite{guide} and references therein. The connection between the fractional Sobolev spaces and the classical Sobolev spaces are now well known. Indeed, in a pioneering work by Bourgain, Brezis, and Mironescu \cite{BBM} (see also Maz'ya-Shaposhnikova \cite{Maz-Sha} for the limit case $s\to 0^+$) it was shown that \begin{equation*}
 \lim_{s\to 1^-}(1-s)[u]^p_{s,p,\rn}=K(p,n)||\nabla u||^p_{L^p(\Omega)}\text{ for }u\in W^{1,p}_0(\Omega), 
\end{equation*}
where the explicit value of the constant $K(p,n)$ (see for instance, \cite{Rossi}) is given by
\begin{equation}\label{constant K}
    K(p,n)=\frac{1}{p}\int_{\mathbb S^{n-1}}|\langle\sigma,\textbf{e}\rangle|^p\;d\mathcal{H}^{n-1}(\sigma)=\frac{2\pi^{\frac{n-1}{2}}\Gamma\big(\frac{p+1}{2}\big)}{p\Gamma\big(\frac{n+p}{2}\big)},\text{ where $\Gamma$ is usual gamma function}. 
\end{equation}
In 2016, Parini and Ruf \cite{Par-Ruf} addressed the fractional Moser-Trudinger inequality in the fractional Sobolev space $\tilde{W}^{s,p}_0(\Omega)$. But their result is not sharp, in the sense that the value of the optimal constant is not known, and remains a problem; an upper bound of the optimal constant is given explicitly. Namely, they proved the following theorem.
\begin{theorem*}[\textbf{A}]
Let $\Omega$ be a bounded open set of $\rn$ $(n\geq 2)$ with Lipschitz boundary, and $0<s<1,\;p>1$ such that $sp=n$. Then there exists $\alpha^*_{s,n}>0$ such that 
\begin{equation}\label{fmt for bdd domain}
     \sup_{\substack{u\in\tilde{W}_0^{s,p}(\Omega),\;[u]_{s,p,\rn}\leq 1}}\int_{\Omega}e^{\alpha|u(x)|^{\frac{n}{n-s}}}\;dx<\infty\;\; \text{ for }\alpha\in[0,\alpha^*_{s,n}).
\end{equation}
Moreover,
\begin{equation*}
   \sup_{\substack{u\in\tilde{W}_0^{s,p}(\Omega),\;[u]_{s,p,\rn}\leq 1}}\int_{\Omega}e^{\alpha|u(x)|^{\frac{n}{n-s}}}\;dx=\infty\;\; \text{ for }\alpha\in(\alpha^*_{s,n},\infty), 
\end{equation*}
where 
\begin{equation}\label{alpha* value}
\alpha^*_{s,n}:=n(\gamma_{s,n})^{\frac{s}{n-s}}=n\bigg(\frac{2\omega_{n-1}^2\Gamma(p+1)}{n!}\sum_{k=0}^{\infty}\frac{(n+k-1)!}{k!(n+2k)^p}\bigg)^{\frac{s}{n-s}},\; \omega_{n-1}=\frac{n\pi^{\frac{n}{2}}}{\Gamma(1+\frac{n}{2})}.
\end{equation}
\end{theorem*}
Iula \cite{Iula} proved the above theorem in one dimension and also discussed some improvement results on the fractional Moser-Trudinger inequality. Note that, in dimension $n=2$ as already mentioned in \cite{Par-Ruf}, $\lim_{\substack{s\to 1^-}}(1-s)\alpha^*_{s,2}=2\pi^2$ which is same as the optimal exponent $\alpha_2=4\pi$ in the local case, up to an appropriate constant which naturally appears in the study of asymptotic behavior of the semi-norm in the limiting case $s\to 1^-$.

We shall study extensions of the fractional Moser-Trudinger inequality \eqref{fmt for bdd domain} to domains in $\rn$ having infinite measure. In this case, it is easy to see that the integral in \eqref{fmt for bdd domain} is infinite, as the integrand is always greater than 1. So an obvious modification is that one should remove some terms of the Taylor series expansion of the exponential function. Namely, for $k\in\mathbb{N}$ we consider the following $k$-th order truncated exponential function
\begin{equation*}\label{xi function}
\Psi_k(x):=e^x-\sum_{j=0}^{k-1}\frac{x^j}{j!},
\end{equation*}
\noindent and if $\Omega\subseteq\rn$ define the fractional Moser's-like functional of order $k$ for $0<\alpha<\alpha^*_{s,n}$  as $$F_{\Omega,k}(u)=\int_{\Omega}\Psi_k(\alpha|u(x)|^{\frac{p}{p-1}})dx.$$
We shall denote the function $\Psi_k$ by $\Psi$ when $k=\ceil{p-1}$ for simplicity of notation, where $\ceil*{\cdot}$ denotes the usual ceiling function i.e. $\ceil*{p-1}$ is the smallest integer which is greater than or equal to $p-1$. We will say that the fractional Moser-Trudinger inequality holds for an unbounded domain $\Omega\subset\rn$ if for some $0<\alpha<\alpha^*_{s,n}$ we have
 \begin{equation}\label{fmt kth order}
 \sup_{\substack{u\in\tilde{W}_0^{s,p}(\Omega),\;[u]_{s,p,\rn}\leq 1}}F_{\Omega,k}(u)<\infty.
 \end{equation}
 In the spirit of the local result (see for instance, \cite{Bat-Man} remark 2.7), we would like to investigate the equivalence between the fractional Moser-Trudinger inequality \eqref{fmt kth order} and the fractional $\left(\frac{kp}{p-1},p\right)$-Poincar\'e type inequality given as follows:
\begin{equation}\label{suitable frac inequality}
    \lambda(\Omega,n,p,k):=\inf_{\substack{u\in\tilde{W}^{s,p}_0(\Omega)\\u\neq 0 }}\frac{\displaystyle\int_{\rn}\int_{\rn}\frac{|u(x)-u(y)|^p}{|x-y|^{2n}}dxdy}{\bigg(\displaystyle\int_{\Omega}|u(x)|^{\frac{kp}{p-1}}dx\bigg)^{\frac{p-1}{k}}}>0.
\end{equation}
For $1<p,q<\infty$, we will say that the fractional $(q,p)$-Poincar\'e type inequality holds if
\begin{align}\label{frac q,p poincare inequality}
 \tilde{\lambda}(\Omega,n,p,q):=\inf_{\substack{u\in\tilde{W}^{s,p}_0(\Omega)\\u\neq 0 }}\frac{\displaystyle\int_{\rn}\int_{\rn}\frac{|u(x)-u(y)|^p}{|x-y|^{2n}}dxdy}{\bigg(\displaystyle\int_{\Omega}|u(x)|^q\;dx\bigg)^{\frac{p}{q}}}>0.
\end{align}

 Now we are in position to state our first main result.
 \begin{theorem}\label{result1}
Let $\Omega\subset\rn$ be an open set and $0<s<1<p<\infty$ with $sp=n$. Then the fractional Moser-Trudinger inequality \eqref{fmt kth order} holds true for $\Omega$ if and only if the fractional $\left(\frac{kp}{p-1},p\right)$-Poincar\'e type inequality \eqref{suitable frac inequality} does.
\end{theorem}
The Schwarz symmetrization plays an essential role in the proof of the above theorem, which we shall be introducing briefly in the forthcoming section. By the Schwarz symmetrization, we can restrict to the non-negative, radially symmetric, non-increasing functions in the space $\tilde{W}^{s,p}_0(\Omega)$. As an application of Theorem \ref{result1}, it is easy to see that the fractional Moser-Trudinger inequality is equivalent to the fractional Poincar\'e inequality (see corollary \ref{FMT eqv to FP}) for a suitable choice of $k.$ Recently for $p=2$, Chowdhury and Roy \cite{Ind-Roy} studied the fractional Poincar\'e inequality for unbounded domains. Since we are using the fractional Poincar\'e inequality extensively in the sequel, we recall the proof of it for general $p$ in Section \ref{sufficient cond}, for the sake of completeness.
\begin{remark}\label{Fp for bdd domain}
When $sp=n$, then by using the standard fractional Sobolev embedding and the fractional Poincar\'e inequality for a bounded domain $\Omega\subset\rn$, we have for any $q>1$ 
\begin{equation*}
    \bigg(\int_{\Omega}|u(x)|^qdx\bigg)^{\frac{p}{q}}\leq C(\Omega,n,p)\int_{\rn}\int_{\rn}
     \frac{|u(x)-u(y)|^p}{|x-y|^{2n}}dxdy\;\;\text{ for all }u\in\tilde{W}^{s,p}_0(\Omega),
\end{equation*} which implies $\tilde{\lambda}(\Omega,n,p,q)>0$. In particular, we have $\lambda(\Omega,n,p,k)>0$ for a bounded domain $\Omega\subset\rn$ and for $k\in\mathbb N.$
\end{remark}
The next aim of the present paper is to study the fractional $(q,p)$-Poincar\'e type inequality \eqref{frac q,p poincare inequality} for unbounded domains in $\rn$. In this direction, we provide a sufficient condition on the domain for which this inequality remains true. To state our next result, we recall a couple of definitions (see also \cite{Ind-Roy}).

\begin{definition}[\textbf{Uniform fractional $(q,p)$-Poincar\'e type inequality}]
Let $\{\Omega_\alpha\}_{\alpha}$ be a family of sets in $\mathbb{R}^n$,
where $\alpha \in \mathbb{A}$ (some indexing set). We say the fractional $(q,p)$-Poincar\'e type inequality to hold uniformly for
 $\{\Omega_\alpha\}_{\alpha}$ if
$\displaystyle\inf_\alpha\tilde{\lambda}(\Omega_\alpha,n,p,q) > 0.$
\end{definition}
Let $P(\omega)$ denote the plane perpendicular to $\omega \in\mathbb S^{n-1}$,
 passing through the origin and for $x_0\in P(\omega)$, define $L_{\Omega}(x_0, \omega) := \left\{t
\ |  \ x_0+t\omega \in \Omega \right\}\subset \mathbb{R}$.

\begin{definition}[\textbf{LS type domain}]\label{defn:LS}  We say $\Omega$
is of type LS if there exists a countable set $\Sigma$ of $\mathbb S^{n-1}$, and for $\omega\in\mathbb{S}^{n-1}\setminus\Sigma$ there exists a set $A(\omega)\subset P(\omega)$ with positive equal finite
$(n-1)$-dimensional Hausdorff measure (i.e. $\mathcal{H}^{n-1}(A(\omega_1))=\mathcal{H}^{n-1}(A(\omega_2))<\infty, \text{ for any } \omega_1,\;\omega_2\in\mathbb S^{n-1}\setminus\Sigma$) such that the one dimensional fractional $(q,p)$-Poincar\'e type inequality holds uniformly for the family of sets
$\left\{L_\Omega(x_0,\omega) \right\}_{x_0\in A(\omega),\;\omega \in
\mathbb S^{n-1}\setminus\Sigma}$.
\end{definition}
\begin{theorem}\label{removed}
Let $\Omega\subset\rn$ be an open set and $s\in(0,1),\;1<p,q<\infty$ with $sp=n$. Then the fractional $(q,p)$-Poincar\'e type inequality \eqref{frac q,p poincare inequality} holds true if $\Omega$ is a LS type domain.
\end{theorem}
The proof of this theorem is the clever use of the standard fractional Sobolev embedding and the one dimensional reduction formula due to Loss-Sloane (\cite{loss}, Lemma 2.4). At the end of Section \ref{sufficient cond}, we shall discuss some examples of LS type domains (see also \cite{Fsk}, \cite{Ind-Roy}). 
\smallskip

In the paper of Parini and Ruf \cite{Par-Ruf} asked whether an inequality of the type 
\begin{equation*}
\sup_{\substack{u\in\tilde{W}_0^{s,p}(\Omega),\;[u]_{s,p,\rn}\leq 1}}\int_{\Omega}f(|u|)e^{{\alpha|u(x)|^{\frac{n}{n-s}}}}\;dx<\infty,   
\end{equation*}
where $f:\re^+\to\re^+$ is such that $f(t)\to\infty$ as $t\to\infty$, holds true for the same exponents of the standard Moser-Trudinger inequality. Iula, in \cite{Iula}, answered this question negatively for $n=1$. Also for the Bessel potential spaces, Iula, Maalaoui, and Martinazzi \cite{Iula-Maa-Mar} $(n=1)$, and Hyder \cite{Hyd} $(n\geq 2)$, it was investigated.
\smallskip

When we move to the entire space that is the case $\Omega=\rn$, the validity of the fractional Moser-Trudinger inequality was studied by Iula \cite{Iula} $(n=1)$ and Zhang \cite{Zha} $(n\geq 2)$. We state the result here of Zhang \cite{Zha}, only for the case required. 
\begin{theorem*}[\textbf{B}]
Let $s\in (0,1)$ and $sp=n$. Then for every $0\leq\alpha<\alpha^*_{s,n}$ it holds
\begin{equation}\label{zhang result}
    FB(n,s,\alpha)=\sup_{\substack{u\in W^{s,p}(\rn)\\||u||_{W^{s,p}(\rn)}\leq 1}}\int_{\rn}\Psi({\alpha|u(x)|^{\frac{n}{n-s}}})\;dx
    \begin{cases}<\infty, & \alpha<\alpha^*_{s,n},\\
    =\infty, & \alpha>\alpha^*_{s,n}.
    \end{cases}
    \end{equation}
\end{theorem*}
Our next theorem verifies the sharpness in the sense of \eqref{borel} of these type inequalities with a larger family of functions for higher dimension. In particular, we have the following:
\begin{theorem}\label{result2}
 Let $\Omega$ be an open bounded subset of $\rn,$ $0<s<1<p$ with $sp=n$. Then we have 
 \begin{equation}\label{borel1}
  \sup_{\substack{u\in\tilde{W}_0^{s,p}(\Omega),\;[u]_{s,p,\rn}\leq 1}}\int_{\Omega}f(|u|)e^{{\alpha^*_{s,n}|u(x)|^{\frac{n}{n-s}}}}\;dx=\infty,   
 \end{equation}
 \begin{equation}\label{borel2}
     \sup_{\substack{u\in W^{s,p}(\rn),\;\ell ||u||_p^p+[u]^p_{s,p,\rn}\leq 1}}\int_{\rn}f(|u|)\Psi({\alpha^*_{s,n}|u(x)|^{\frac{n}{n-s}}})\;dx=\infty,
 \end{equation}
 where $f:\re^+\to\re^+$ is any Borel measurable function such that $f(t)\to\infty\text{ as }t\to\infty$ and for any $\ell>0$.
 \end{theorem}
In our next result, we deal with the Adachi-Tanaka type Moser-Trudinger inequality in the fractional Sobolev space. Indeed, we prove the following:
\begin{theorem}\label{result3}
Let $0<s<1<p$ with $sp=n$. There exists $\alpha^*_{s,n}>0$ such that
\begin{equation*}
    FA(n,s,\alpha)=\sup_{\substack{u\in W^{s,p}(\rn)\;[u]_{s,p,\rn}\leq 1}}\frac{1}{||u||^p_{L^p(\rn)}}\int_{\rn}\Psi({\alpha|u(x)|^{\frac{n}{n-s}}})\;dx
    \begin{cases}<\infty, & 0<\alpha<\alpha^*_{s,n},\\
    =\infty, & \alpha\geq\alpha^*_{s,n}.
    \end{cases}
\end{equation*}
\end{theorem}
In recent day research, many authors have shown their interest in the study of fractional Moser-Trudinger inequality. We refer to \cite{Martin} for the study of fractional Adams-Moser-Trudinger inequality in the case of Bessel potential spaces. For fractional Moser-Trudinger inequalities with singular weight we refer to \cite{Thin}. Reader may refer to \cite{tyagi}, \cite{Per-Marco} for some of the relevant research in this direction and references therein for the available literature in this direction.

\section{Some preliminary and known results}
We fix some notions that will be used in the entire article. Given a measurable set $\Omega\subset\rn$, $|\Omega|$ denotes the Lebesgue measure. An open ball in $\rn$ centered at $x$ with radius $r$ will be denoted by $B(x,r)$ and the $(n-1)$-dimensional Hausdorff measure will be denoted by $\mathcal{H}^{n-1}.$ We start with some known facts and some technical lemmas that will be essentially required to proof our main results. As we mentioned in the introduction, let us recall the Schwarz symmetrization argument.
\\ Given a measurable set $E\subset\rn$ with finite Lebesgue measure, denote $E^*$, the symmetric rearrangement of $E$ is the open ball centered at the origin with equal measure as $E$, i.e.
$$
E^*=\{x\in\rn:|x|<r\}, \text{ with }r^n=\frac{n|E|}{\omega_{n-1}}.
$$
For a measurable function $u:\rn\to\re$ such that $|\{x:|u(x)|>t\}|<\infty$, for all $t>0,$ define $u^*$, the symmetric decreasing rearrangement of $u$ by 
$$
u^*(x)=\int_0^\infty\displaystyle\chi_{\{|u|>t\}^*}(x)dt=\sup\Big\{t:|\{|u|>t\}|>\frac{\omega_{n-1}|x|^n}{n}\Big\}.
$$
It is worth mentioning that the function $u^*$ is non-negative, radially symmetric and decreasing.  We refer to \cite{hichem}, \cite{Kesavan}, \cite{Lieb-Loss} for more details about this topic. We recall two important properties that will be useful in the proof of Theorem \ref{result1}.
\begin{proposition}\label{pnorm}
Given a Borel measurable function $F:\re\to\re$ and a non-negative function $u:\rn\to\re$, then we have $$\int_{\rn}F(u(x))dx=\int_{\rn}F(u^*(x))dx.$$
\end{proposition}
The following result is about the P\'olya-Szeg\"o type inequality in the fractional setting and it can be found in \cite{Alm-Lieb}.
\begin{theorem}\label{seminorm}
Let $0<s<1\leq p<\infty$, and $u\in W^{s,p}(\rn)$. Then we have 
$$
[|u|^*]_{s,p,\rn}\leq[u]_{s,p,\rn}.
$$
\end{theorem}

We recall the following definition (see also \cite{Fsk}).

\begin{definition}[\textbf{Finite ball condition}]
\label{def:finite ball cond}
We say that a set $\Omega\subset\rn$ satisfies the \textit{finite ball condition} if $\Omega$ does not contain arbitrarily large balls, that is
$$
   \sup\{ r:\, B(x,r)\subset\Omega,\; x\in \Omega\}<\infty.
$$
\end{definition}

\begin{lemma}\label{fbc not}
Let $\Omega$ be an open set in $\rn$. Then the fractional $(q,p)$-Poincar\'e type inequality \eqref{frac q,p poincare inequality} is not true if the domain $\Omega$ does not satisfy the finite ball condition.
\end{lemma}
\begin{proof}
We first fix, $0\neq u\in C_c^\infty(B(0,1))$ and define $M:=\frac{\int_{\rn}\int_{\rn}\frac{|u(x)-u(y)|^p}{|x-y|^{2n}}dxdy}{||u||_{L^q(B(0,1))}^p}$. By remark \ref{Fp for bdd domain} we have $0<M<\infty.$ Since the domain $\Omega$ does not satisfy the finite ball condition, then for any $\ell>0$ sufficiently large there exists $x_\ell\in\Omega$ such that $B(x_\ell,\ell)\subset\Omega.$ Defining $v_\ell(x)=u\big(\frac{x-x_\ell}{\ell}\big)$, it is immediate to see that 
\begin{equation*}
    \tilde{\lambda}(\Omega,n,p,q)\leq\frac{[v_\ell]^p_{s,p,\rn}}{||v_\ell||^p_{L^q(\Omega)}}=\ell^{-\frac{np}{q}}M\to 0,\text{ as }\ell\to\infty.
\end{equation*}
This proves the lemma.
\end{proof}
The following result is essentially contained in \cite{Bras-Lin-Par}, but we include the proof of it with a slight modification.
\begin{lemma}\label{strong fp}
Let $\Omega$ be a bounded open subset of $\rn$ and $1\leq p<\infty$, $\sigma>0$. Then 
$$
\int_{\Omega}|u(x)|^p dx\leq \frac{\textnormal{diam}(\Omega\cup B_{R})^{n+\sigma}}{|B_{R}|}\int_{\Omega}\int_{B_{R}}\frac{|u(x)-u(y)|^p}{|x-y|^{n+\sigma}}dydx\;\text{ for all }u\in C_c^\infty(\Omega),
$$
where $B_{R}\subset\rn\setminus\Omega$ is a ball of radius $R$.
\end{lemma}
\begin{proof}
Since the domain $\Omega$ is bounded, we can always find a ball $B_{R}$ of radius $R$ such that $B_{R}\subset\rn\setminus\Omega.$ Let $u\in C_c^\infty(\Omega)$ and for any $x\in\Omega$, $y\in B_{R}$ we have
$$
|u(x)|^p=\frac{|u(x)-u(y)|^p}{|x-y|^{n+\sigma}}|x-y|^{n+\sigma},
$$
then integrating over $B_{R}$ with respect to $y$ we get
\begin{align*}
|u(x)|^p=\frac{1}{|B_{R}|}\int_{B_{R}}\frac{|u(x)-u(y)|^p}{|x-y|^{n+\sigma}}&|x-y|^{n+\sigma}dy\\
&\leq\frac{\text{diam}(\Omega\cup B_{R})^{n+\sigma}}{|B_{R}|}\int_{B_{R}}\frac{|u(x)-u(y)|^p}{|x-y|^{n+\sigma}}dy,
\end{align*}
and integrating again over $\Omega$ we acquire that 
$$
\int_{\Omega}|u(x)|^p dx\leq\frac{\text{diam}(\Omega\cup B_{R})^{n+\sigma}}{|B_{R}|}\int_{\Omega}\int_{B_{R}}\frac{|u(x)-u(y)|^p}{|x-y|^{n+\sigma}}dydx.
$$
This completes the proof. 
\end{proof}
\begin{remark}
For a bounded domain $\Omega\subset\rn$ and $s\in(0,1),\;1\leq p<\infty$ with $sp=n$, then by the standard fractional Sobolev embedding and the above lemma we obtain, for any $q\geq 1$
\begin{align*}
     \bigg(\int_{\Omega}|u(x)|^qdx\bigg)^{\frac{p}{q}}\leq C(\Omega,n,p)\bigg(\int_{\Omega}&\int_{\Omega}
     \frac{|u(x)-u(y)|^p}{|x-y|^{2n}}dxdy\\
     & +\int_{\Omega}\int_{B_{R}}
     \frac{|u(x)-u(y)|^p}{|x-y|^{n+\sigma}}dydx\bigg)\;\;\text{ for all }u\in\tilde{W}^{s,p}_0(\Omega),
\end{align*}
where $B_{R}$ is the same as in Lemma \ref{strong fp}; this implies
\begin{align*}
    \Lambda(\Omega,n,p,q):=\inf_{\substack{u\in\tilde{W}^{s,p}_0(\Omega)\\u\neq 0 }}\frac{\int_{\Omega}\int_{\Omega}\frac{|u(x)-u(y)|^p}{|x-y|^{2n}}dxdy+\int_{\Omega}\int_{B_{R}}
     \frac{|u(x)-u(y)|^p}{|x-y|^{n+\sigma}}dydx}{||u||_{L^q(\Omega)}^p}>0.
\end{align*}
\end{remark}
The proof of the following two results can be found in \cite{Man}, but for the sake of simplicity we recall the proofs here.
\begin{lemma}\label{bound for psi}
If $k\geq 1$ and the function $\Psi_k$ which is defined in \eqref{xi function} then for any $M>0$ one has $$\Psi_k(z)\leq Cz^k\;\;\text{ for all }0\leq z\leq M$$ for some constant $C=C(k,M).$
\end{lemma}
\begin{proof}
We proceed by induction on $k$, since for $k\geq 2$ we have 
$$
\Psi_k^{\prime}(z)=e^z-\sum_{j=1}^{k-1}\frac{z^{j-1}}{(j-1)!}=\Psi_{k-1}(z).
$$
Now, for $k=1$
$$
\Psi_1(z)=e^z-1\leq ze^z\leq e^M z.
$$
Suppose that the result holds for $k-1$, then
$$
\Psi_k(z)=\int_0^z\Psi_{k-1}(t)dt\leq C(k,M)\int_0^z t^{k-1}dt=\frac{C(k,M)}{k}z^k.
$$
\end{proof}
\begin{lemma}\label{gen radial lemma}
If $k\geq 1$ and $u\in L^{\frac{kp}{p-1}}(\rn)$ is an non-negative radially decreasing function, then for any $x\neq 0$

$$
u(x)\leq\Big(\frac{n}{\omega_{n-1}}\Big)^{\frac{p-1}{kp}}||u||_{\frac{kp}{p-1}}\frac{1}{|x|^{\frac{n(p-1)}{kp}}}.
$$
\end{lemma}
\begin{proof}
We apply the H\"older inequality with the exponents $p_1=\frac{kp}{p-1}$ and $p_2=\frac{kp}{p(k-1)+1}$
\begin{align}
    R^n u(R)=nu(R)\int_{0}^{R}t^{n-1}dt
    &\leq n\int_{0}^{R}t^{n-1}u(t)dt\nonumber\\
    &\leq n\bigg(\int_{0}^{R}t^{n-1}u(t)^{\frac{kp}{p-1}}dt\bigg)^{\frac{p-1}{kp}}\bigg(\int_{0}^{R}t^{n-1}dt\bigg)^{\frac{1}{p_2}}\nonumber\\
    &=n^{\frac{1}{p_1}}R^{\frac{n}{p_2}}\bigg(\frac{1}{\omega_{n-1}}\int_{|x|<R}u(x)^{\frac{kp}{p-1}}dx\bigg)^{\frac{p-1}{kp}}\nonumber.
\end{align}
Thus, we get $u(x)\leq\Big(\frac{n}{\omega_{n-1}}\Big)^{\frac{p-1}{kp}}||u||_{\frac{kp}{p-1}}\frac{1}{|x|^{\frac{n(p-1)}{kp}}}$ which proves the lemma.
\end{proof}
\section{Proof of Theorem \ref{result1}}
In this section, we prove Theorem \ref{result1} and at the end of the section, we address some consequence results of the fractional Moser-Trudinger inequality.
\smallskip

\noindent\textit{\textbf{Proof of Theorem \ref{result1}:}}
Suppose the fractional Moser-Trudinger inequality \eqref{fmt kth order} holds. Then we have 
\begin{equation}\label{fmt1}
\int_{\Omega}\Psi_k(\alpha |v(x)|^{\frac{p}{p-1}})<\infty
\end{equation}
where $v(x)=\frac{u(x)}{[u]_{s,p,\rn}}$ for $u\in \tilde{W}_0^{s,p}(\Omega).$ Now by definition of the function $\Psi_k$ and by \eqref{fmt1} we obtain $$\int_{\Omega}|v(x)|^{\frac{kp}{p-1}}<C(n,p)\implies\bigg(\int_{\Omega}|u(x)|^{\frac{kp}{p-1}}dx\bigg)^{\frac{p-1}{k}}\leq C(n,p)\;[u]^p_{s,p,\rn}.$$
Therefore we conclude that the fractional $\left(\frac{kp}{p-1},p\right)$-Poincar\'e type inequality holds true.\smallskip

To prove the other case, we use basically the Schwarz symmetrization. By exploiting Proposition \ref{pnorm} and Theorem \ref{seminorm}, it is enough to show the result for non-negative, radially decreasing functions $u\in\tilde{W}_0^{s,p}(\Omega)$ with the semi-norm $[u]_{s,p,\rn}\leq 1.$ We write 
 \begin{equation}\label{split domain}
     \int_{\rn}\Psi_k(\alpha(u)^{\frac{n}{n-s}})=\int_{B(0,r_0)}\Psi_k(\alpha(u)^{\frac{n}{n-s}})+\int_{B(0,r_0)^c}\Psi_k(\alpha(u)^{\frac{n}{n-s}})=:I_1+I_2,
 \end{equation}
where the values of $r_0>1$ to be chosen later.
We estimate the integrals $I_1$ and $I_2$ in the following:

\noindent\textit{Estimate of $I_2$:} By using Lemma \ref{gen radial lemma} and by the fractional $\left(\frac{kp}{p-1},p\right)$-Poincar\'e type inequality \eqref{suitable frac inequality}, we conclude that $u(x)$ is bounded for $|x|>r_0.$ Thus, applying Lemma \ref{bound for psi} we infer that 
$$\Psi_k(\alpha u^{\frac{p}{p-1}})\leq C(n,p,k)u^{\frac{kp}{p-1}},$$
and therefore 
\begin{align}\label{estimate of I2}
    I_2=\int_{B(0,r_0)^c}\Psi_k(\alpha(u)^{\frac{p}{p-1}})
    &\leq C(n,p,k)\int_{B(0,r_0)^c}|u(x)|^{\frac{kp}{p-1}}dx\leq\frac{C(n,p,k)}{\lambda(\Omega,n,p,k)},
\end{align}
where in the last estimate again we use the fractional $\left(\frac{kp}{p-1},p\right)$-Poincar\'e type inequality \eqref{suitable frac inequality}.\\
\noindent\textit{Estimate of $I_1$:} Consider the function 
\begin{equation*}
    v(x)=\begin{cases}u(x)-u(r_0) & |x|\leq r_0,\\
                      0 & |x|>r_0.
    \end{cases}
\end{equation*}
Let $x\in B(0,r_0)$ and using the decreasing property of the function $u$ we have
\begin{align}\label{seminorm v1}
    \int_{\rn}\frac{|v(x)-v(y)|^p}{|x-y|^{2n}}dy
    &=\int_{B(0,r_0)}\frac{|v(x)-v(y)|^p}{|x-y|^{2n}}dy+\int_{B(0,r_0)^c}\frac{|v(x)-v(y)|^p}{|x-y|^{2n}}dy\nonumber\\
    &=\int_{B(0,r_0)}\frac{|u(x)-u(y)|^p}{|x-y|^{2n}}dy+\int_{B(0,r_0)^c}\frac{|u(x)-u(r_0)|^p}{|x-y|^{2n}}dy\nonumber\\
    &\leq\int_{\rn}\frac{|u(x)-u(y)|^p}{|x-y|^{2n}}dy.
\end{align}
Similarly, one can show for $x\in B(0,r_0)^c$ we have 
\begin{align}\label{seminorm v2}
   \int_{\rn}\frac{|v(x)-v(y)|^p}{|x-y|^{2n}}dy\leq\int_{\rn}\frac{|u(x)-u(y)|^p}{|x-y|^{2n}}dy. 
\end{align}
Combining \eqref{seminorm v1}, \eqref{seminorm v2} and then integrating in the variable $x$, we obtain 
\begin{equation}\label{seminorm comparision}
    [v]_{s,p,\rn}^p\leq[u]_{s,p,\rn}^p.
\end{equation}
Now using the elementary inequality $(a+b)^q\leq a^q+q2^{q-1}(a^{q-1}b+b^q)$ for $a,b\geq 0,\;q\geq 1$ and by definition of the function $v$ we obtain
\begin{align}\label{final inq}
    u(x)^{\frac{n}{n-s}}
    &\leq v(x)^{\frac{n}{n-s}}+\frac{n}{n-s}2^{\frac{s}{n-s}}(v(x)^{\frac{s}{n-s}}u(r_0)+u(r_0)^{\frac{n}{n-s}})\nonumber\\
    &\leq v(x)^{\frac{n}{n-s}}\bigg(1+\frac{n^{\frac{p-1}{k}2^{\frac{1}{p-1}}}}{(p-1)\;\omega_{n-1}^{\frac{p-1}{k}}\;r_0^{\frac{n(p-1)}{k}}}||u||^p_{\frac{kp}{p-1}}\bigg)+C\nonumber\\
    &=:v(x)^{\frac{n}{n-s}}\big(1+\beta||u||_{\frac{kp}{p-1}}^p\big)+C.
\end{align}
This implies
\begin{align*}
    u(x)
    &\leq v(x)\big(1+\beta||u||_{\frac{kp}{p-1}}^p\big)^{\frac{n-s}{n}}+C^{\frac{n-s}{n}}\\
    &=:w(x)+C^{\frac{n-s}{n}}.
\end{align*}
It is worth mentioning that by using the hypothesis in this case and Lemma \ref{gen radial lemma}, the above constant $C$ is independent of $u$. Let $\theta(\Omega,n,p,k)=\frac{1}{\lambda(\Omega,n,p,k)}>0$ and since $||u||_{\frac{kp}{p-1}}^p+[u]_{s,p,\rn}^p\leq 1+\theta(\Omega,n,p,k)\implies [u]_{s,p,\rn}^p\leq 1+\theta(\Omega,n,p,k)-||u||_{\frac{kp}{p-1}}^p$. Therefore, using the semi-norm estimates \eqref{seminorm comparision} we obtain
\begin{align}
    [w]_{s,p,\rn}^p
    &=[v]_{s,p,\rn}^p\big(1+\beta||u||_{\frac{kp}{p-1}}^p\big)^{\frac{n-s}{s}}\nonumber
    &\leq\big(1+\theta(\Omega,n,p,k)-||u||_{\frac{kp}{p-1}}^p\big)\big(1+\beta||u||_{\frac{kp}{p-1}}^p\big)^{\frac{n-s}{s}}.
    \end{align}
Now we define a function $f(t)=(1+\theta(\Omega,n,p,k)-t)(1+\beta t)^{\sigma}$, where $\sigma:=\frac{n-s}{s}>0.$ Then 
$$
f^\prime(t)=(1+\beta t)^{\sigma-1}((1+\theta(\Omega)-t))\sigma\beta-1-\beta t).
$$
So $f^\prime(t)=0$ for $t_1=-\frac{1}{\beta}<0$, and for $t_2=\frac{(1+\theta(\Omega,n,p,k)\sigma\beta-1}{\beta(1+\sigma)}$. To make $t_2<0$ choose $r_0^n>\frac{n(1+\theta(\Omega,n,p,k)^{\frac{k}{p-1}}2^{\frac{k}{(p-1)^2}}}{\omega_{n-1}}$. We get this choice by looking at the right hand side of the expression of $t_2$. Hence $f$ is non-increasing in $(0,1+\theta(\Omega,n,p,k))$ and since $f(1+\theta(\Omega,n,p,k))=0$, we get $f(t)<1+\theta(\Omega,n,p,k)$ for $t\in(0,1+\theta(\Omega,n,p,k))$ and therefore we have 
\begin{equation*}
 [w]_{s,p,\rn}^p\leq 1+\theta(\Omega,n,p,k)\implies\frac{[w]_{s,p,\rn}^p}{1+\theta(\Omega,n,p,k)}\leq 1.  
 \end{equation*}
 By using the fractional Moser-Trudinger inequality on the ball $B(0,r_0)$, we infer that there exists $\alpha\in(0,\alpha^*_{s,n})$ such that 
 \begin{equation*}
     \int_{B(0,r_0)}e^{\alpha_1 w^{\frac{n}{n-s}}}\leq C,
 \end{equation*}
 where $\alpha_1:=\frac{\alpha}{(1+\theta(\Omega,n,p,k))^{\frac{n}{n-s}}}<\alpha.$ By using \eqref{final inq} we obtain 
 \begin{equation}\label{estimate of I1}
     I_1\leq\int_{B(0,r_0)}e^{\alpha_1 u^{\frac{n}{n-s}}}\leq C_1\int_{B(0,r_0)}e^{\alpha_1 w^{\frac{n}{n-s}}}\leq C_2.
 \end{equation}
Combining \eqref{split domain}, \eqref{estimate of I2}, \eqref{estimate of I1} we get the desired result.
\smallskip
\begin{corollary}\label{FMT eqv to FP}
Let $\Omega$ be an open set in $\rn$ and $0<s<1$, $p\in\mathbb N\setminus\{1\}$ with $sp=n$. Then the fractional Moser-Trudinger inequality holds for $\Omega$ if and only if the fractional Poincar\'e inequality does. 
\end{corollary}
\begin{proof}
It follows from the above theorem by choosing $k=p-1$.
\end{proof}

\begin{proposition}
Let $\Omega$ be a bounded open subset of $\rn$ and $0<s<1<p$ with $sp=n$, $\alpha\in(0,\alpha^*_{s,n})$. Then for any $\lambda\in[0,(\frac{p\alpha}{p-1})^{p-1}]$ there exists a constant $C\in\re$ such that 
\begin{align}\label{onofri type}
    \frac{1}{p}\int_{\rn}\int_{\rn}\frac{|u(x)-u(y)|^p}{|x-y|^{2n}}dxdy-\lambda\log\bigg(\frac{1}{|\Omega|}\int_{\Omega}e^{u(x)}dx\bigg)\geq C\;\;\text{ for all }u\in\tilde{W}_0^{s,p}(\Omega).
\end{align}
\end{proposition}
\begin{proof}
Let $u\in\tilde{W}_0^{s,p}(\Omega)$ and applying Young's inequality with the exponents $p$ and $q=\frac{n}{n-s}$ to get
\begin{multline*}
    |u(x)|=\left(\frac{1}{q\alpha}\right)^{\frac{1}{q}}[u]_{s,p,\rn}\;(q\alpha)^{\frac{1}{q}}\frac{|u(x)|}{[u]_{s,p,\rn}}\leq\left(\frac{1}{q\alpha}\right)^{\frac{p}{q}}\frac{[u]_{s,p,\rn}^p}{p}+\alpha\left(\frac{|u(x)|}{[u]_{s,p,\rn}}\right)^q\\
    =:A+\alpha\left(\frac{|u(x)|}{[u]_{s,p,\rn}}\right)^q.
\end{multline*}
Using the fact that the function $e^t$ is increasing and then integrating over $\Omega$ we conclude that
\begin{equation*}
\frac{1}{|\Omega|}\int_{\Omega}e^{u(x)}dx\leq\frac{e^A}{|\Omega|}\int_{ \Omega}e^{\alpha\left(\frac{|u(x)|}{[u]_{s,p,\rn}}\right)^q}\leq C e^{\left(\frac{1}{q\alpha}\right)^{\frac{p}{q}}\frac{[u]_{s,p,\rn}^p}{p}}.
\end{equation*}
In the last inequality, we used the fractional Moser-Trudinger inequality. Thus, we have 
$$
\log\bigg(\frac{1}{|\Omega|}\int_{\Omega}e^{u(x)}dx\bigg)\leq\frac{(q\alpha)^{-\frac{p}{q}}}{p}[u]_{s,p,\rn}^p+\log C
$$
which implies the desired inequality \eqref{onofri type}.
\end{proof}
\begin{lemma}[Asymptotic behavior ]
For $0<s<1<p$ with $sp=n$ and the constant $K(p,n)$ is given by \eqref{constant K} then one has
$$
\lim_{s\to 1^-}\frac{(1-s)\gamma_{s,n}}{K(p,n)}=\omega_{n-1},
$$
where $\gamma_{s,n}$ is appearing in \eqref{alpha* value}. 
\end{lemma}
\begin{proof}
By definition of the constant $K(p,n)$ in \eqref{constant K} and thanks to Proposition 5.1 in \cite{Par-Ruf} gives the required result.
\end{proof}

\section{On Sufficient condition}\label{sufficient cond}
In this section we prove Theorem \ref{removed}. As we mentioned in the introduction, we first prove the fractional Poincar\'e inequality. For this, let us define  
\begin{align*}
    P^2_{n,s,p}(\Omega):=\inf_{\substack{u\in\tilde{W}^{s,p}_0(\Omega)\\u\neq 0 }}\frac{\displaystyle\int_{\rn}\int_{\rn}\frac{|u(x)-u(y)|^p}{|x-y|^{n+sp}}dxdy}{||u||_{L^p(\Omega)}^p}.
\end{align*}
We will say that fractional Poincar\'e inequality holds for $\Omega\subset\rn$ if $P^2_{n,s,p}(\Omega)>0$. It is well known that $P^2_{n,s,p}(\Omega)>0$, for any bounded set $\Omega\subset\rn$ (see \cite{Bras-Lin-Par}). In the following result, the fractional Poincar\'e inequality remains true for an unbounded domain in $\re$. For this,
let $E=\cup_{k=1}^\infty I_k\subset\re$ be an open set which satisfies the following:
\begin{equation}\label{set 1D}
	dist(I_k,I_{k+1})=m>0,\;\;\sup_{k}(|I_k|)=M<\infty,
\end{equation}
where $I_k,\;I_{k+1}$ be any two consecutive open intervals. 
\begin{lemma}\label{1D FP}
Let $E\subset\re$ be as above. Let $s\in(0,1)$, $1<p<\infty$. Then the fractional Poincar\'e inequality holds for $E$.
\end{lemma}
\begin{proof}
Let us choose $\tilde{E}:=\bigcup_{j=1}^{\infty}\tilde{I_j}\subset E^c$ be an open set such that $\tilde{I_j}$'s are disjoint open interval with $\inf_{j}(\tilde{I_j})=m_1>0$. Denoting $\tilde{I}_{j(k)}$ is an interval in between the intervals $I_k,\;I_{k+1}$. Let $u\in C_c^\infty(E)$.
\begin{multline*}
    \int_{\re}\int_{\re}\frac{|u(x)-u(y)|^p}{|x-y|^{1+sp}}dxdy
    \geq\int_{E}\int_{E^c}\frac{|u(x)-u(y)|^p}{|x-y|^{1+sp}}dydx
    \geq\int_{E}\int_{\tilde{E}}\frac{|u(x)-u(y)|^p}{|x-y|^{1+sp}}dydx\\
    =\sum_{k,j=1}^{\infty}\int_{I_k}\int_{\tilde{I_j}}\frac{|u(x)-u(y)|^p}{|x-y|^{1+sp}}dydx
    \geq\sum_{k=1}^{\infty}\int_{I_k}\int_{\tilde{I}_{j(k)}}\frac{|u(x)-u(y)|^p}{|x-y|^{1+sp}}dydx\\
    \geq\frac{m_1}{(M+m)^{1+sp}}\int_{E}|u(x)|^pdx,
\end{multline*}
which gives $P^2_{1,s,p}(E)>0$. This completes the proof of the lemma.
\end{proof}
The next lemma holds true for any $p>0$ but for the purpose of our analysis we state it for $p>1.$
\begin{lemma}[Loss-Sloane \cite{loss}, Lemma 2.4]\label{loss- sloane lemma}
Let $\Omega\subset\rn$ be any set and $s\in(0,1),\;p>1$. Then for any $u\in C_c^\infty(\Omega)$
\begin{align*}
     &2\int_{\Omega}\int_{\Omega}
     \frac{|u(x)-u(y)|^p}{|x-y|^{n+sp}}dxdy
     \smallskip
     \\
    =&\int_{\mathbb S^{n-1}}d\mathcal{H}^{n-1}(\omega)
    \int_{\{x:\;x\cdot\omega=0\}}d\mathcal{H}^{n-1}(x)
    \int_{\{x+\ell\omega\in\Omega\}}
    \int_{\{x+t\omega\in\Omega\}}\frac{|u(x+\ell \omega)-u(x+t\omega)|^p}{|\ell-t|^{1+sp}}dtd\ell.
\end{align*}
\end{lemma}
\smallskip
\begin{theorem}\label{fracp inq}
Let $\Omega\subset\rn$ be an open set. Then the fractional Poincar\'e inequality holds true if $\Omega$ is a LS type domain.
\end{theorem}
\begin{proof}
Taking $\Omega=\rn$ in Lemma \ref{loss- sloane lemma}. Then, we obtain
\begin{align*}
     &2\int_{\rn}\int_{\rn}
     \frac{|u(x)-u(y)|^p}{|x-y|^{n+sp}}dxdy\nonumber
     \\
    \geq&\int_{\mathbb S^{n-1}\setminus\Sigma}d\mathcal{H}^{n-1}(\omega)
    \int_{A(\omega)\subset{\{x:\;x\cdot\omega=0\}}}d\mathcal{H}^{n-1}(x)
    \int_{\re}
    \int_{\re}\frac{|u(x+\ell\omega)-u(x+t\omega)|^p}{|\ell-t|^{1+sp}}dtd\ell.
\end{align*}
Now since the domain satisfies the LS condition, this implies that for each fixed $\omega\in\mathbb S^{n-1}\setminus\Sigma$, $x\in A(\omega)$ there exists a constant say $C>0$ independent of $\omega$ and $A(\omega)$ such that
\begin{align*}
    \int_{\re}\int_{\re}\frac{|u(x+\ell\omega)-u(x+t\omega)|^p}{|\ell-t|^{1+sp}}dtd\ell\geq C\int_{\re}|u(x+t\omega)|^p dt.
\end{align*}
Combining the above two estimates we infer that
\begin{align*}
    \int_{\rn}\int_{\rn}
     \frac{|u(x)-u(y)|^p}{|x-y|^{n+sp}}dxdy\geq C_1\int_{\Omega}|u(x)|^p.
\end{align*}
This completes the proof of the theorem.
\end{proof}
\begin{proposition}
Let $E\subset\re$ be as in \eqref{set 1D} and $s\in(0,1)$, $1<p,q<\infty$ with $sp=1$. Then the fractional $(q,p)$-Poincar\'e type inequality holds for $E.$
\end{proposition}
\begin{proof}
Let $\tilde{E}\subset E^c$ is same as in the proof of Lemma \ref{1D FP}.
Let $u\in C_c^\infty(E)$, and using the fractional Sobolev embedding and Lemma \ref{strong fp} we acquire that
\begin{align*}
   \int_{\re}\int_{\re}
     \frac{|u(x)-u(y)|^p}{|x-y|^{2}}& dxdy
     \geq\int_{E}\int_{E}
     \frac{|u(x)-u(y)|^p}{|x-y|^{2}}dxdy+\int_{E}\int_{\tilde{E}}
     \frac{|u(x)-u(y)|^p}{|x-y|^{2}}dydx\\
     &\geq\sum_{k,\;j=1}^\infty\bigg(\int_{I_k}\int_{I_k}
     \frac{|u(x)-u(y)|^p}{|x-y|^{2}}dxdy+\int_{I_k}\int_{\tilde{I_j}}
     \frac{|u(x)-u(y)|^p}{|x-y|^{2}}dydx\bigg)\\
     &\geq\sum_{k=1}^\infty\bigg(\int_{I_k}\int_{I_k}
     \frac{|u(x)-u(y)|^p}{|x-y|^{2}}dxdy+\int_{I_k}\int_{\tilde{I}_{j(k)}}
     \frac{|u(x)-u(y)|^p}{|x-y|^{2}}dydx\bigg)\\
     &\geq\sum_{k=1}^\infty\Lambda(I_k,1,p,q)\Big(\int_{I_k}|u(x)|^q dx\Big)^{\frac{p}{q}}\\
     &=\sum_{k=1}^\infty\frac{\Lambda((0,1),1,p,q)}{|I_k|^{\frac{p}{q}}}\Big(\int_{I_k}|u(x)|^q dx\Big)^{\frac{p}{q}}\geq\frac{C}{M^{\frac{p}{q}}}\Big(\int_{E}|u(x)|^q dx\Big)^{\frac{p}{q}},
\end{align*}
which gives $\tilde{\lambda}(E,1,p,q)>0$. This completes the proof of the proposition.  
\end{proof}

\noindent\textit{\textbf{Proof of Theorem \ref{removed} :}}
\noindent\textit{Step 1:} Since the fractional Poincar\'e inequality holds for LS type domain $\Omega$ and then applying this to the standard fractional Sobolev embedding, we acquire that $\tilde{\lambda}(\Omega,n,p,q)>0$ for any $q\in[p,\infty).$
\smallskip

\noindent\textit{Step 2:} Taking $\Omega=\rn$ in Lemma \ref{loss- sloane lemma}, we obtain
\begin{align}\label{loss-sloane result}
     &2\int_{\rn}\int_{\rn}
     \frac{|u(x)-u(y)|^p}{|x-y|^{2n}}dxdy\nonumber
     \\
    \geq&\int_{\mathbb S^{n-1}\setminus\Sigma}d\mathcal{H}^{n-1}(\omega)
    \int_{A(\omega)\subset{\{x:\;x\cdot\omega=0\}}}d\mathcal{H}^{n-1}(x)
    \int_{\re}
    \int_{\re}\frac{|u(x+\ell\omega)-u(x+t\omega)|^p}{|\ell-t|^{1+n}}dtd\ell.
\end{align}
Note that, for any bounded open set $E\subset\re$ 
\begin{align*}
   &\int_{\re}\int_{\re}\frac{|u(x+\ell\omega)-u(x+t\omega)|^p}{|\ell-t|^{1+n}}dtd\ell\\
   &\geq\frac{1}{\text{diam}(E)^{n-1}}
   \int_{E}
   \int_{E}\frac{|u(x+\ell\omega)-u(x+t\omega)|^p}{|\ell-t|^{2}}dtd\ell+
   \int_{E}\int_{B_{R}\subset E^c}\frac{|u(x+\ell\omega)-u(x+t\omega)|^p}{|\ell-t|^{1+n}}dtd\ell. 
\end{align*}
Now since we have assumed the domain satisfying LS condition, this gives that for each fixed $\omega\in\mathbb S^{n-1}\setminus\Sigma$, $x\in A(\omega)$ there exists a constant say $C>0$ independent of $\omega$ and $A(\omega)$ such that
\begin{align}\label{critical embedding}
    \int_{\re}\int_{\re}\frac{|u(x+\ell\omega)-u(x+t\omega)|^p}{|\ell-t|^{1+n}}dtd\ell\geq C\bigg(\int_{\re}|u(x+t\omega)|^q dt\bigg)^{\frac{p}{q}}.
\end{align}
Combining the above two estimates \eqref{loss-sloane result}, \eqref{critical embedding} we infer that
\begin{align*}
     \int_{\rn}\int_{\rn}
     \frac{|u(x)-u(y)|^p}{|x-y|^{2n}}dxdy\geq\frac{C}{2}\int_{\mathbb S^{n-1}\setminus\Sigma}d\mathcal{H}^{n-1}(\omega)
    \int_{A(\omega)}d\mathcal{H}^{n-1}(x)\bigg(\int_{\re}|u(x+t\omega)|^q dt\bigg)^{\frac{p}{q}},
\end{align*} and then by using Jensen's inequality we get
\begin{align*}
     \int_{\rn}\int_{\rn}
     \frac{|u(x)-u(y)|^p}{|x-y|^{2n}}dxdy
     &\geq C_1\bigg(\int_{\mathbb S^{n-1}\setminus\Sigma}d\mathcal{H}^{n-1}(w)
    \int_{A(\omega)}d\mathcal{H}^{n-1}(x)\int_{\re}|u(x+t\omega)|^q dt\bigg)^{\frac{p}{q}}\\
    &=C_1\bigg(\int_{\Omega}|u(x)|^q dx\bigg)^{\frac{p}{q}}.
\end{align*}
Thus we conclude that $\tilde{\lambda}(\Omega,n,p,q)>0$, for any $q\in(1,p].$ This completes the proof of the theorem.
\smallskip

As an application of Theorem \ref{removed} we will give some examples of unbounded domains for which the fractional $(q,p)$-Poincar\'e type inequality \eqref{frac q,p poincare inequality} holds true. However, verification of the hypothesis of Theorem \ref{removed} in the below examples is straight forward.
\begin{example}
(1)\textit{ Domain between graph of functions:} Let $u_i:\re^{n-1}\to[a,b]\;(i=1,2)$ be two bounded continuous function such that $u_1<u_2.$ The domain $\Omega$ is defined as
$$
\Omega=\{(x,t)\in\rn:u_1(x)<t<u_2(x)\}.
$$

(2) \textit{Infinitely many parallel strips:} Let $\Omega:=\left(\bigcup_{k=1}^\infty I_k\right)\times\re^{n-1}\subset\rn$, where $I_k$'s are disjoint open intervals in $\re$ with $\mbox{dist}(I_i, I_j)\geq\alpha>0$ for any $i\neq j$ and a constant $\alpha$. Also assume that $\mbox{sup}_{k}(|I_k|)=M<\infty.$
\end{example}
\section{Proof of Theorem \ref{result2}}
In this section, we prove Theorem \ref{result2} following the same techniques used in \cite{Iula}.
\smallskip

\noindent\textit{\textbf{Proof of Theorem \ref{result2} :}}
 Without loss of generality, we can take $\Omega$ is the unit ball $B_1$. Consider the usual Moser-sequence of functions which is defined by
 \begin{equation}\label{moser sequence}
     u_{\epsilon}(x)=\begin{cases}|\log\epsilon|^{\frac{n-s}{n}} & \text{ if }|x|\leq\epsilon,\\
       \frac{|\log|x||}{|\log\epsilon|^{\frac{s}{n}}} & \text{ if }\epsilon<|x|<1,\\
       0 & \text{ if } |x|\geq 1.
     \end{cases}
 \end{equation}
 Clearly as $\epsilon\to 0$, $u_\epsilon\to\infty$ uniformly for $|x|<\epsilon$. Now we have 
 \begin{equation*}
     \sup_{\substack{u\in\tilde{W}_0^{s,p}(\Omega),\;[u]_{s,p,\rn}\leq 1}}\int_{\Omega}f(|u|)e^{{\alpha^*|u(x)|^{\frac{n}{n-s}}}}\;dx\geq\inf_{|x|<\epsilon}f(|u_\epsilon|)\int_{|x|<\epsilon}e^{\alpha^*_{s,n}\big(\frac{|u_\epsilon|}{[u_\epsilon]_{s,p,\rn}}\big)^{\frac{n}{n-s}}}dx.
 \end{equation*}
By hypothesis on the function $f$, it is suffices to check that there exists $\delta>0$ such that 
 \begin{equation*}\label{seminorm estimate}
     \int_{|x|<\epsilon}e^{\alpha^*_{s,n}\big(\frac{|u_\epsilon|}{[u_\epsilon]_{s,p,\rn}}\big)^{\frac{n}{n-s}}}dx\geq\delta.
 \end{equation*}
 In view of Section 5 of \cite{Par-Ruf}, we see that 
 
 $$
 \lim_{\epsilon\to 0}\frac{[u_\epsilon]_{s,p,\rn}^p}{\gamma_{s,n}}=1
 $$
and in particular we have

$$
 \lim_{\epsilon\to 0}[u_\epsilon]_{s,p,\rn}^p=C(n,p)\int_1^\infty |\log x|^p x^{n-1}\frac{x^2+1}{(x^2-1)^{n+1}}dx=\gamma_{s,n}.
$$
Now 
\begin{equation*}
\lim_{\epsilon\to 0}\log\frac{1}{\epsilon}\left([u_\epsilon]_{s,p,\rn}^p-\gamma_{s,n}\right)=C(n,p)\lim_{\epsilon\to 0}\log\frac{1}{\epsilon}\int_{1/\epsilon}^\infty |\log x|^p x^{n-1}\frac{x^2+1}{(x^2-1)^{n+1}}dx=0.
\end{equation*}
Thus for small enough $\epsilon>0$ we have
 \begin{equation}\label{u_epsilon estimate}
     \frac{[u_\epsilon]_{s,p,\rn}^p}{\gamma_{s,n}}\leq 1+\frac{C}{\log\frac{1}{\epsilon}}.
 \end{equation}
 Noticing that 
 \begin{equation}\label{elementary limit}
     \lim_{t\to\infty}\bigg(\frac{t}{(1+\frac{C}{t})^{\frac{1}{p-1}}}-t\bigg)=-\frac{C}{p-1}.
 \end{equation}
 Thus, using \eqref{u_epsilon estimate} and \eqref{elementary limit} we get 
 \begin{align}\label{epsilon ball estimate of moser functional}
     \int_{|x|\leq\epsilon}e^{\alpha^*_{s,n}\big(\frac{|u_\epsilon|}{[u_\epsilon]_{s,p,\rn}}\big)^{\frac{n}{n-s}}}dx
     =\int_{|x|\leq\epsilon}e^{n\big(\frac{\gamma_{s,n}^s}{[u_\epsilon]_{s,p,\rn}^n}\big)^{\frac{1}{n-s}}|u_\epsilon|^{\frac{n}{n-s}}}&dx
     \geq\int_{|x|<\epsilon}e^{\frac{n\log\frac{1}{\epsilon}}{(1+\frac{C}{\log\frac{1}{\epsilon}})^{\frac{1}{p-1}}}}dx\nonumber\\
     &=C_1(n)\epsilon^n\;e^{\frac{n\log\frac{1}{\epsilon}}{(1+\frac{C}{\log\frac{1}{\epsilon}})^{\frac{1}{p-1}}}}\to e^{-\frac{nC}{p-1}}
 \end{align}
 as $\epsilon\to 0.$
 Therefore, we conclude that 
 \begin{equation}\label{lower bound of moser functional}
  \int_{|x|<\epsilon}e^{\alpha^*_{s,n}\big(\frac{|u_\epsilon|}{[u_\epsilon]}\big)^{\frac{n}{n-s}}}dx\geq\delta,
 \end{equation}
 for some $\delta>0$. This conclude the proof of the first part \eqref{borel1} of the theorem. In order to prove the second result \eqref{borel2} of the theorem, one can conclude it from the first result \eqref{borel1} of the theorem by the following observation. For $\ell>0$ and $u\in W^{s,p}(\rn)$ with $\ell||u||_p^p+[u]_{s,p,\rn}^p\leq 1$, we define a function $u_{\ell}(x)=u(\ell^{1/n}x)$. Then by the change of variable we acquire that $||u_\ell||_p^p+[u_\ell]_{s,p,\rn}^p=\ell||u||_p^p+[u]_{s,p,\rn}^p$, and
 \begin{align*}
 \sup_{\substack{u\in W^{s,p}(\rn),\;\ell ||u||_p^p+[u]^p_{s,p,\rn}\leq 1}}&\int_{\rn}f(|u|)\Psi({\alpha^*|u(x)|^{\frac{n}{n-s}}})\;dx\\
 &=\frac{1}{\ell}\sup_{\substack{u\in W^{s,p}(\rn),\;||u||_{W^{s,p,}(\rn)}\leq 1}}\int_{\rn}f(|u|)\Psi({\alpha^*|u(x)|^{\frac{n}{n-s}}})\;dx.
 \end{align*}
 Now by definition \eqref{moser sequence} of $u_\epsilon$ we have
 \begin{equation}\label{lpnorm}
     ||u_\epsilon||_p^p=\int_{\rn}|u_\epsilon(x)|^p dx=\int_{|x|\leq\epsilon}|\log\epsilon|^{\frac{n-s}{s}}dx+\int_{\epsilon<|x|<1}\frac{|\log|x||^p}{|\log\epsilon|}dx=O\bigg(\frac{1}{\log\frac{1}{\epsilon}}\bigg),
 \end{equation}
 and thanks to Proposition 5.1 in \cite{Par-Ruf}, we infer that 
 \begin{equation*}
     \lim_{\epsilon\to 0}||u_\epsilon||^p_{W^{s,p}(\rn)}=\gamma_{s,n}.
 \end{equation*}
 From \eqref{u_epsilon estimate} and \eqref{lpnorm} we obtain 
 \begin{equation}\label{full norm estimate}
     \frac{||u_\epsilon||^p_{W^{s,p}(\rn)}}{\gamma_{s,n}}\leq 1+O\bigg({\frac{1}{\log\frac{1}{\epsilon}}}\bigg).
 \end{equation}
 Note that for sufficiently large enough $M>0$ such that $\Psi(t)\geq\frac{e^t}{2}$, $t\geq M.$ Thus, for $\epsilon>0$ small enough we can write
 \begin{align*}
     \int_{\rn}\Psi\bigg(n(\gamma_{s,n})^{\frac{s}{n-s}}\bigg(\frac{u_\epsilon(x)}{||u_\epsilon||_{W^{s,p}(\rn)}}\bigg)^{\frac{n}{n-s}}\bigg)dx
     &\geq\int_{u_\epsilon\geq M}\Psi\bigg(n(\gamma_{s,n})^{\frac{s}{n-s}}\bigg(\frac{u_\epsilon(x)}{||u_\epsilon||_{W^{s,p}(\rn)}}\bigg)^{\frac{n}{n-s}}\bigg)dx\nonumber\\
     &\geq\frac{1}{2}\int_{|x|\leq\epsilon}e^{\big(\frac{\gamma_{s,n}}{||u_\epsilon||^p_{W^{s,p}(\rn)}}\big)^{\frac{1}{p-1}}n u_\epsilon^{\frac{n}{n-s}}}dx.
 \end{align*}
Now we can proceed exactly in the same way as it is done in \eqref{epsilon ball estimate of moser functional}, \eqref{lower bound of moser functional} together with \eqref{full norm estimate} to obtain the required result. 
\section{Proof of Theorem \ref{result3}}
In this section, we start with giving some technical lemmas that will be required to the proof of Theorem \ref{result3}. We shall follow the approach by Takahashi \cite{Tak}.
\begin{lemma}\label{technical lemma 1}
Set 
\begin{equation*}
    FA_1(n,s,\alpha)=\sup_{\substack{u\in W^{s,p}(\rn)\\ [u]_{s,p,\rn}\leq 1,\;||u||_{L^p(\rn)}^p}=1}\int_{\rn}\Psi({\alpha|u(x)|^{\frac{n}{n-s}}})\;dx.
\end{equation*}
Then $FA_1(n,s,\alpha)=FA(n,s,\alpha)$ for any $\alpha>0.$
\end{lemma}
\begin{proof}
For $\ell>0$ and $u\in W^{s,p}(\rn)$ with $[u]_{s,p,\rn}\leq 1$, we define a function $u_{\ell}(x)=u(\ell x)$. Then by simple computations we get $||u_\ell||_{L^p(\rn)}^p=\ell^{-n}||u||_{L^p(\rn)}^p$, $[u_\ell]^p_{s,p,\rn}=[u]^p_{s,p,\rn}\leq 1$. Thus, if we choose $\ell^n=||u||^p_{L^p(\rn)}$, then $u_\ell\in W^{s,p}(\rn)$ satisfies
$$
[u_\ell]_{s,p,\rn}\leq 1\text{ and }||u_\ell||^p_{L^p(\rn)}=1.
$$
Therefore we have 
$$
\frac{1}{||u||^p_{L^p(\rn)}}\int_{\rn}\Psi({\alpha|u(x)|^{\frac{n}{n-s}}})\;dx=\int_{\rn}\Psi({\alpha|u_\ell(x)|^{\frac{n}{n-s}}})\;dx\leq FA_1(n,s,\alpha)
$$
which gives $FA(n,s,\alpha)\leq FA_1(n,s,\alpha)$ and the other inequality trivially follows.
\end{proof}
\begin{lemma}\label{techinal lemma 2}
Let $\alpha_\epsilon=\alpha^*_{s,n}-\epsilon$, for fixed $\epsilon>0$ small enough. Then for any $\alpha\in(0,\alpha_\epsilon)$ one has 
$$FA(n,s,\alpha)\leq\frac{(\frac{\alpha}{\alpha_\epsilon})^{p-1}}{1-(\frac{\alpha}{\alpha_\epsilon})^{p-1}}FB(n,s,\alpha_\epsilon).$$
\end{lemma}
\begin{proof}
Let $u\in W^{s,p}(\rn)$ with $||u||_{L^p(\rn)}=1$ and $[u]_{s,p,\rn}\leq 1.$ Then, we consider a function $v(x)=Cu(\ell x)$, where $C^p=(\frac{\alpha}{\alpha_\epsilon})^{p-1}$ and $\ell^n=\frac{C^p}{1-C^p}.$ By simple computation we obtain $||v||^p_{L^p(\rn)}=C^p\ell^{-n}$ and $[v]^p_{s,p,\rn}=C^p[u]^p_{s,p,\rn}$. Therefore we have
\begin{equation*}
    ||v||^p_{W^{s,p}(\rn)}=||v||^p_{L^p(\rn)}+[v]^p_{s,p,\rn}\leq C^p(1+\ell^{-n})=1.
\end{equation*}
Now 
\begin{align*}
    FB(n,s,\alpha_\epsilon)\geq\int_{\rn}\Psi({\alpha_\epsilon|v(x)|^{\frac{n}{n-s}}})\;dx
    &=
    \int_{\rn}\Psi({\alpha_\epsilon\;C^{\frac{p}{p-1}}|u(\ell x)|^{\frac{n}{n-s}}})\;dx\\
    &=
    \ell^{-n}\int_{\rn}\Psi({\alpha|u(x)|^{\frac{n}{n-s}}})\;dx.
\end{align*}
Taking the supremum for $u\in W^{s,p}(\rn)$ with $||u||_{L^p(\rn)}=1$ and $[u]_{s,p,\rn}\leq 1$, together with Lemma \ref{technical lemma 1} gives the desired result.
\end{proof}
\smallskip

\noindent\textit{\textbf{Proof of Theorem \ref{result3} :}}  It is easy to see that $FA(n,s,\alpha)<\infty$ for $\alpha<\alpha^*_{s,n}$ by using Lemma \ref{techinal lemma 2} and \eqref{zhang result}.\\
In order to proof of $FA(n,s,\alpha^*_{s,n})=\infty.$ We use the Moser-sequence of functions which is defined in \eqref{moser sequence}. Then we have the following estimates as $\epsilon\to 0$, for some $C>0$
\begin{align*}
    ||u_\epsilon||^p_{L^p(\rn)}=O\bigg(\frac{1}{\log\frac{1}{\epsilon}}\bigg),\\
    [u_\epsilon]^p_{s,p,\rn}=\gamma_{s,n}+O(1),\\
     \frac{[u_\epsilon]^p_{s,p,\rn}}{\gamma_{s,n}}\leq 1+\frac{C}{\log\frac{1}{\epsilon}}.
\end{align*}
Now set $v_\epsilon(x)=\frac{u_\epsilon(x)}{[u_\epsilon]_{s,p,\rn}}$, then by using the above estimates we acquire 
\begin{align*}
    FA(n,s,\alpha^*_{s,n})
    \geq\frac{1}{||v_\epsilon||^p_{L^p(\rn)}}\int_{\rn}& \Psi({\alpha^*_{s,n}|v_\epsilon(x)|^{\frac{n}{n-s}}})\;dx\nonumber\\
    &\geq\frac{[u_\epsilon]^p_{s,p,\rn}}{2||u_\epsilon||^p_{L^p(\rn)}}\int_{|x|\leq\epsilon}e^{\big(\frac{\gamma_{s,n}}{||u_\epsilon||^p_{W^{s,p}(\rn)}}\big)^{\frac{1}{p-1}}n u_\epsilon^{\frac{n}{n-s}}}dx\nonumber\\
    &\geq\frac{\gamma_{s,n}+O(1)}{O\big(\frac{1}{\log\frac{1}{\epsilon}}\big)}\int_{|x|\leq\epsilon}e^{\frac{n\log\frac{1}{\epsilon}}{(1+\frac{C}{\log\frac{1}{\epsilon}})^{\frac{1}{p-1}}}}dx\nonumber\\
    &=C_1(n)\frac{\gamma_{s,n}+O(1)}{O\big(\frac{1}{\log\frac{1}{\epsilon}}\big)}\;\;e^{\frac{n\log\frac{1}{\epsilon}}{(1+\frac{C}{\log\frac{1}{\epsilon}})^{\frac{1}{p-1}}}-n\log\frac{1}{\epsilon}}\to\infty
\end{align*}
as $\epsilon\to 0.$ This completes the proof of the theorem.
\smallskip

\textbf{Acknowledgement:} The author would like to thank his advisor Prof. Prosenjit Roy for his encouragement on the subject. The author would also like to thank Prof. Gyula Csat\'o for fruitful discussions and comments on this research.

\end{document}